\documentclass[reqno]{amsart}
\usepackage{amsmath,amssymb,latexsym}
\usepackage{mathpazo}
\usepackage[mathscr]{eucal}
\usepackage{enumerate}
\usepackage[all]{xy}

\newcommand{\calW}{{\mathcal{W}}}

\newcommand{\calC}{\mathcal{C}}
\newcommand{\calF}{\mathcal{F}}

\newcommand{\calO}{\mathcal{O}}
\newcommand{\calE}{\mathcal{E}}
\newcommand{\calJ}{\mathcal{J}}
\newcommand{\calR}{\mathcal{R}}
\newcommand{\calS}{\mathcal{S}}

\newcommand{\frq}{\mathfrak{q}}

\newcommand{\frakm}{\mathfrak{m}}

\newcommand{\C}{\mathbb{C}}
\newcommand{\R}{\mathbb{R}}

\newcommand{\N}{\mathbb{N}}

\newcommand{\frakg}{\mathfrak{g}}

\newcommand{\fWeyl}{\mathbb W}

\newcommand{\Sym}{\operatorname{Sym}}
\newcommand{\cSym}{\operatorname{\widehat{Sym}}}

\newcommand{\supp}{\operatorname{supp}}

\newcommand{\jet}{\mathsf{J}}

\numberwithin{equation}{section}
\theoremstyle{plain}
        \newtheorem{theorem}{Theorem}[section]
        \newtheorem{lemma}[theorem]{Lemma}
        \newtheorem{proposition}[theorem]{Proposition}

\theoremstyle{definition}
        \newtheorem{definition}[theorem]{Definition}
        \newtheorem{remark}[theorem]{Remark}
        \newtheorem{example}[theorem]{Example}

\title{Quantization of Whitney functions and reduction}
\author{M.J.~Pflaum, H. Posthuma,~\textrm{and} X.~Tang}
\begin{document}
\begin{abstract} 
 For a possibly singular subset of a regular Poisson manifold we construct a deformation quantization
 of its algebra of Whitney functions. We then extend the construction of a deformation quantization 
 to the case where the underlying set is a subset of a not necessarily regular Poisson
 manifold which can be written as the quotient of a regular Poisson manifold on which a compact Lie group 
 acts freely by Poisson maps. Finally, if the quotient Poisson manifold is regular as well, we show
 a ``quantization commutes with reduction'' type result. 
 For the proofs, we use methods stemming from both singularity theory and Poisson geometry. 
\end{abstract}
\dedicatory{The paper is dedicated to David Trotman  on the occasion of his 60ies birthday.}
\address{\newline
Markus J. Pflaum, {\tt markus.pflaum@colorado.edu}\newline
         \indent {\rm Department of Mathematics, University of Colorado,
         Boulder, USA}\newline
        Hessel Posthuma, {\tt h.b.posthuma@uva.nl}\newline
         \indent {\rm Korteweg-de Vries Institute for Mathematics, University
         of Amsterdam, The Netherlands} \newline
        Xiang Tang, {\tt xtang@math.wustl.edu}   \newline
         \indent {\rm  Department of Mathematics, Washington University,
         St.~Louis, USA}}
\maketitle
\section*{Introduction}
In this paper we consider the synthesis of two, seemingly different, branches of mathematics, namely 
that of singularity theory and Poisson geometry and deformation quantization. There are motivations from 
both sides to consider such a blend: from the point of view of Poisson geometry and mathematical 
physics, singularities naturally appear when one considers Poisson manifolds with symmetries of which one wants to take the quotient. From the point of view of singularity theory, the general idea that
a quantization can act as a kind of ``noncommutative desingularization'' has had quite a few striking 
applications. To make proper sense of this idea one needs to combine this with techniques coming from 
noncommutative geometry.

In this paper we use the notion of Whitney functions to describe  the deformation quantization of a (singular) set inside a Poisson manifold.  More specifically, we describe how the Fedosov method 
applies to construct such deformation quantizations inside a regular Poisson manifold, and prove a
``quantization commutes with reduction'' type of result for the quantized Whitney functions invariant 
under a free action of a compact Lie group that preserves the Poisson structure. 
\section{Formal deformation quantizations of  Whitney functions}
Recall that for a closed subset $X\subset M$ of a smooth manifold $M$ the 
algebra of Whitney 
functions on $X$ is defined as the quotient  
$\calE^\infty (X;M) := \calC^\infty (M) / \calJ^\infty (X,M)$, where  
\[
  \calJ^\infty (X,M) := \big\{ f \in \calC^\infty (M) \mid (Df)_{|X} = 0  
  \text{ for every differential operator $D$ on $M$}\big\}
\] 
denotes the ideal of smooth functions on $M$ which are flat  on $X$. 
If no confusion about the ambient space can arise, we briefly write $\calE^\infty (X)$ 
instead of $\calE^\infty (X;M)$. Moreover, we denote the canonical quotient 
map from $\calC^\infty (M)$ to $\calE^\infty (X;M)$, sometimes called the \emph{jet map}, 
by $\jet_{X;M}$ or $\jet_X$, if no confusion can ariese.
Finally observe that if $\Phi : M \rightarrow N$ is a smooth map between manifolds $M$ and $N$ 
which maps the closed subset $X\subset M$ into a closed subset $Y\subset N$, then there is 
a canonical \emph{pull-back map} for Whitney functions
\[
  \Phi^* : \calE^\infty (Y;N) \rightarrow \calE^\infty (X;M), 
\]
which maps the Whitney function $F = \jet_{Y;N} (f)$ represented by $f \in \calC^\infty (N)$ to the 
Whitney function $\jet_{X;M} \big( f\circ \Phi\big) $.  The reader will easily 
check that the pull-back  is well-defined.

Recall further that by a \emph{Whitney--Poisson} structure  on $X$ one understands 
a bilinear map $\{ - , - \}$ on $\calE^\infty (X)$ which satisfies  for all $F,G,H \in \calE^\infty (X)$
the following relations
\begin{enumerate}[(WP1)]
  \item \label{IteAntisymmetry}
  $\{F, G\} = - \{ G , F\}$,
  \item \label{IteDerivation}
  $\{F, GH\}=\{F,G\}H+G\{F,H\}$,  and 
  \item \label{IteWPJacobi}
  $\{\{F,G\}, H\}+\{\{H,F\},G\}+\{\{G,H\},F\}=0$. 
\end{enumerate}
In other words, (WP\ref{IteAntisymmetry}) tells that $\{ - , - \}$ is an antisymmetric bilinear form, 
(WP\ref{IteDerivation}) says that $\{ - , - \}$ is a derivation in each of its arguments, 
and (WP\ref{IteWPJacobi}) is the Jacobi identity. Hence there exists a smooth antisymmetric bivector field
$ \Lambda : X \rightarrow TM \otimes TM$ such that 
\[
  \{F, G\}  = \Lambda \, \lrcorner \, (dF \otimes dG ) \quad \text{for all $F,G \in \calE^\infty (X)$}.
\]
Note that we have used here the fact that $\calJ^\infty (X,M) \Omega^\bullet (M)$ is a graded 
ideal in  $\Omega^\bullet (M)$ preserved by the exterior derivative $d$ which gives rise to the differential graded 
quotient algebra  $\Omega^\bullet_{\calE^\infty} (X) := \Omega^\bullet (M)/\calJ^\infty (X;M) \Omega^\bullet (M)$.
Its differential will be denoted again by $d$. We call $\Omega^\bullet_{\calE^\infty} (X) $ the complex of 
Whitney--de Rham forms  on $X$. According to \cite{BraPfl}, the cohomology of 
$\Omega^\bullet_{\calE^\infty} (X)$ coincides with the singular cohomology (with values in $\R$), if $M$ 
is an analytic manifold, and $X\subset M$ a subanalytic subset. 
Now we have the means to define what one understands by a formal deformation quantization of 
the algebra of Whitney functions.
\begin{definition} 
  Assume that $X \subset M$ is a closed subset of the smooth manifold $M$, and that 
  $\calE^\infty (X)$ carries a Whitney--Poisson structure. 
  By a \emph{formal deformation quantization} of the algebra $\calE^\infty (X)$  or in other words 
  by a \emph{star product} on $\calE^\infty (X)$ one then understands an associative product 
  \[ 
   \star : \calE^\infty (X)[[\hbar]] \times \calE^\infty (X)[[\hbar]] \rightarrow \calE^\infty (X)[[\hbar]] 
  \]
  on the space $\calE^\infty (X)[[\hbar]]$ of formal power series in the variable $\hbar$ with coefficients 
  in $\calE^\infty (X)$ such that the following is satisfied:
  \begin{enumerate}[(DQ1)]
  \setcounter{enumi}{-1}
  \item
      The product $\star $ is $\R[[\hbar]]$-linear and $\hbar$-adically continuous in each argument.
  \item
     There exist $\R$-bilinear operators $c_k : \calE^\infty (X) \times \calE^\infty (X) \rightarrow \calE^\infty (X)$, $k\in \N$ 
     such that $c_0 $ is the standard commutative product on $\calE^\infty (X)$ and such that for all $F,G \in \calE^\infty (X)$ 
     there is an expansion of the product $ F \star G $ of the form
     \begin{equation}
     \label{eq:starprodexp}
       F \star G = \sum_{k\in \N} c_k (F,G) \hbar^k .
     \end{equation}
  \item 
     The constant function $1 \in \calE^\infty$ satisfies $1 \star F = F \star 1 = F$  for all $F \in \calE^\infty (X)$.
  \item
     The star commutator $[ F, G]_\star := F\star G - G \star F$ of two Whitney functions $F,G \in \calE^\infty (X)$
     satisfies the  commutation relation 
     \[
        [ F, G]_\star = - i \hbar \{ F, G\} + o(\hbar^2) .
     \]
  \end{enumerate}
    If in addition the  condition
   \begin{enumerate}[(DQ1)]
   \setcounter{enumi}{3}
  \item 
    $\supp (F \star G) \subset \supp (F) \cap \supp (G)$ for all $F,G \in \calE^\infty (X)$,
  \end{enumerate}
    is satisfied,  then the star product is called \emph{local}. 
\end{definition}
\begin{remark}
  If $\Pi$ is a Poisson bivector on the smooth manifold $M$, then the ideal $\calJ^\infty (X;M)$ is even a Poisson 
  ideal in $\calC^\infty (M)$. This implies that the Poisson bracket on $\calC^\infty (M)$ factors to the 
  quotient $\calE^\infty (X)$. 
  We denote the inherited Poisson bracket on $\calE^\infty (X)$ also by $\{ - , - \}$, and call it 
  a \emph{global Whitney--Poisson structure}.  
\end{remark}
\vspace{2mm} 

Now let us describe a method for constructing a formal deformation quantization  of the algebra 
$\calE^\infty (X)$ in case 
$(M,\Pi)$ is a regular Poisson manifold 
and  $\calE^\infty (X)$ carries the corresponding global Whitney--Poisson structure. 
This method generalizes the original construction by Fedosov \cite{FedDQIT} to the Whitney function case,
and has been explained in detail by the authors in \cite{PflPosTanQWF} for the particular case where 
the Poisson bivector comes from a symplectic structure. 
Recall that $(M,\Pi)$ being a regular Poisson manifold means that the Poisson tensor field 
$\Pi : M \rightarrow TM \otimes TM$ has constant rank; see \cite{FedDQIT,VaiLGPM} for 
more details on regular Poisson manifolds. Moreover, regularity of $\Pi$ implies that $M$ is 
foliated in a natural way by symplectic manifolds. Denote by $\calS$ the foliation of $M$ by symplectic 
leaves which is induced by the regular Poisson tensor $\Pi$, and by $T\calS\rightarrow M$ the 
subbundle of $TM$ of all tangent vectors tangent to the symplectic leaves of the foliation. 
The following result then holds true. For its original proof we refer to Fedosov \cite{FedDQIT};
here we present a proof which also covers the later needed case of a regular Poisson manifold 
with a compatible $G$-action.  

\begin{proposition}[cf.~{\cite[Sec.~5.7]{FedDQIT}}]
\label{PropPoissonConn}
 For every regular Poisson manifold $(M,\Pi)$ there exists \emph{Poisson connection} which means a 
 connection 
 \[
   \nabla : \Gamma^\infty (T \calS) \rightarrow  \Gamma^\infty (T\calS \otimes T^*\calS)
 \]
 which leaves the Poisson bivector $\Pi$ invariant in the sense that 
 \[
   \nabla \Pi  = 0 .
 \]
 Moreover, if a compact Lie group $G$ acts on $M$ by Poisson maps, the Poisson connection
 $\nabla$ can be chosen to be invariant. 
\end{proposition}

\begin{proof}
  Choose a riemannian metric $\eta$ on $M$ which is required to be $G$-invariant, if $M$ carries a 
  $G$-action compatible with the Poisson structure. 
  Denote by
  \[
   \nabla^\textup{LC} : \Gamma^\infty (T\calS) \times \Gamma^\infty (T\calS) \rightarrow \Gamma^\infty (T\calS)
  \]
  the leafwise Levi--Civita connection of the riemannian metric restricted to $\calS$. 
  Moreover, let $\omega : T\calS \otimes T\calS \rightarrow  \R$ be the leafwise symplectic structure 
  induced by the Poisson bivector. Now we define a tensor field 
  $\Delta' \in \Gamma^\infty (T^*\calS \otimes T^*\calS \otimes T^*\calS )$ by 
  \[
    \Delta' (X,Y,Z) = \nabla^\textup{LC}_Z \omega (X,Y) - \nabla^\textup{LC}_Y \omega (X,Z) 
    \quad \text{for all $X,Y,Z \in \Gamma^\infty (T\calS)$}. 
  \]
  We then let $\Delta \in \Gamma^\infty (T^*\calS \otimes T^*\calS \otimes T\calS )$ be the tensor field 
  such that 
  \[
   \omega \big( X, \Delta (Y,Z) \big) =  \Delta' (X,Y,Z)  
   \quad \text{for all $X,Y,Z \in \Gamma^\infty (T\calS)$}. 
  \]
  By construction it is clear that $\Delta'$ and $\Delta$ are both $G$-invariant, if $\Pi$ 
  and $\eta$ (and hence $\omega$) are. 
  Now we put
  \[
    \nabla_X Y := \nabla^\textup{LC}_X Y + \Delta (X,Y) \quad \text{for  $X,Y \in \Gamma^\infty (T\calS)$}.
  \]
  One readily checks that $\nabla$ is a Poisson connection, and $G$-invariant, if 
  $\Pi$ and $\eta$ are. 
\end{proof}

Next, we consider the  Weyl algebra bundle $\fWeyl_\calS M  \rightarrow M$ over $M$ along the symplectic 
foliation $\calS$.
Its typical fiber over $p\in M$ is given by
\[
  \fWeyl_{\calS,p} M := \fWeyl  (T_p\calS ) := \cSym (T^*_p\calS)[[\hbar]],
\]
the space of formal power series in $\hbar$ with coefficients
in the space of  Taylor expansions at the origin of smooth functions on the fiber $S_p$ of $S$ over $p$.
In other words, $\cSym (T^*_p\calS)$ coincides with the $\frakm$-adic completion of the space 
$\Sym (T^*_p\calS)$ of polynomial functions on $T_pS$, where $\frakm $ denotes the maximal 
ideal in $\Sym (T^*_p\calS)$.
Hence, every element $a$ of $\fWeyl  (T_p \calS) $ can be uniquely expressed in the form
\begin{equation}
\label{eq:fWeylrep}
  a = \sum_{s\in \N , \, k\in \N} a_{s,k} \hbar^k ,
\end{equation}
where each $a_{s,k}$ is an element of $\Sym^s (T^*_p\calS)$, which can be naturally identified with 
the space of $s$-homogeneous polynomial functions on $T_p\calS$. 
%Note that the  $a_{s,k}$ are uniquely defined by $a$.
A section $a \in \calW_\calS (M) := \Gamma^\infty (\fWeyl_\calS M)$ can be uniquely written in the form 
\eqref{eq:fWeylrep}, where the $a_{s,k}$ with $s,k\in\N$  now are smooth sections of the symmetric powers 
$\Sym^s (T^*\calS)$. This representation allows us to define the \emph{symbol map} 
$\sigma: \calW_\calS \rightarrow \calC^\infty (M)[[\hbar]]$ by 
\begin{equation}\label{eq:symbol}
  \sigma (a ) = \sum_{k\in \N} a_{0,k} \hbar^k \quad \text{for $a\in \calW$}.
\end{equation}
The space $\fWeyl  (T_p \calS ) $ is filtered by the \emph{Fedosov-degree}
\[
  \deg_\textup{F} (a) := \min\{s +2k \mid a_{s,k}\neq 0 \}, \quad 
  a \in \fWeyl  (T_p \calS). 
\] 
The Fedosov-degree induces a filtration of the space of sections  $\calW_\calS (M)$ of the Weyl algebra 
bundle along $\calS$ by putting 
\[
  \calF^k \calW_\calS (M) := 
   \{ a \in \calW (M) \mid \deg_\textup{F} (a(p)) \geq k \text{ for all $p \in M$}  \} .
\]
Now consider $\Omega^\bullet \fWeyl_\calS  $, the sheaf of leafwise smooth differential forms with  values 
in the bundle $\fWeyl_\calS M$, or in other words the sheaf of smooth sections of the (profinite 
dimensional) vector bundle $\fWeyl_\calS M \otimes \Lambda^\bullet T^*\calS$. Like $\calW_\calS (M)$, 
the space $\Omega^\bullet \calW_\calS (M)$ is also filtered by the Fedosov-degree. 
%More precisely, it is given by 
%\[
%\begin{split}
%  \calF^k \Omega^\bullet \fWeyl_\calS (M) := 
%   \Big\{ & \sum_{i=1}^l a_i \omega_i \in \Omega^\bullet \fWeyl_\calS  \mid a_i \in \calW_\calS (M), \:
%   \omega_i \in\Gamma^ \infty( \Lambda^\bullet S^*), \\
%   & \text{ and } \deg_\textup{F} (a_i (p)) \geq k \text{ for all $i$ and $p \in M$}  \Big\} .
%\end{split}
%\]

Next, we define a non-commutative algebra structure on $\calW_\calS (M)$ and $\Omega^\bullet\calW_\calS (M)$.
%the so-called \emph{Moyal product}, see \cite{BayetalDTQIuII}. 
To this end observe first that the Poisson bivector $\Pi (p) $ on $T_pM$ is linear and can be 
written in the form  
\begin{equation}
\label{eq:Poissonrep}
  \Pi (p) = \sum_{i=1}^{\frac{\dim S_p}{2}} \Pi_{i1}(p) \otimes \Pi_{i2} (p),
\end{equation}
where $\Pi_{i1}(p),\Pi_{i2} (p)\in T_p\calS$ for $i=1,\cdots , \operatorname{rk} (\Pi)$.
Since each of the tangent vectors $\Pi_{i1}(p),\Pi_{i2} (p)$ acts as a derivation on $\Sym (T^*_p\calS)$,
this gives rise to the operator
\begin{equation}
\begin{split}
  \widehat{\Pi} (p) : \: & \Sym (T^*_p\calS) \otimes\Sym (T^*_p\calS) \rightarrow 
  \Sym (T^*_p\calS) \otimes \Sym (T^*_p\calS), \\ 
  & a \otimes b \mapsto \sum_{i=1}^{\frac{\operatorname{rk} (\Pi) }{2} } \Pi_{i1}(p) \cdot a \otimes \Pi_{i2} (p) \cdot b .
\end{split}
\end{equation}
Th operator $\widehat{\Pi} (p)$ does not depend on the particular representation \eqref{eq:Poissonrep}. 
Note that by $\C[[\hbar]]$-linearity and $\frakm$-adic continuity, $\widehat{\Pi}$ uniquely extends to an operator
\[
 \widehat{\Pi} (p) :  \cSym (T^*_p\calS)[[\hbar]] \otimes \cSym (T^*_p\calS)[[\hbar]] 
  \rightarrow  \cSym (T^*_p\calS)[[\hbar]] \otimes \cSym (T^*_p\calS)[[\hbar]] .
\] 
%The product of two elements $a, b \in \fWeyl  (T_p M ) $ can now be defined. 
The so-called \emph{Moyal--Weyl product} (see \cite{BayetalDTQIuII}) of two elements 
$a, b\in \fWeyl(S_p)$ and is given  by
\begin{equation}
  a \circ_p b := \sum \frac{ (- i \hbar)^k}{k!} \mu \big( \widehat{\Pi} (p)
  ( a \otimes b ) \big) .
\end{equation}
One checks easily that $\circ_p$ is a star product on $\fWeyl  (S_p)$. 
Moreover, this fiberwise star product extends naturally to a noncommutative product $\star$ on
$\calW_\calS (M)$, called the \emph{Moyal--Weyl product} on the Weyl algebra bundle. 
For $a, b \in \calW_\calS  (M) $ it is given by  
\begin{equation}
  \label{Eq:MoyWeylWAB}
  a \circ b (p) :=  a(p) \circ_p b (p) \quad \text{for $p\in M$}.
\end{equation}
Note that the Moyal--Weyl product on $\calW_\calS (M) $ satisfies by construction 
\begin{equation}
  \label{Eq:CommRelMoyProd}
  [ a, b]_\circ : = a \circ b - b \circ a = - i \hbar \{ a , b \} + o(\hbar^2) \: . 
\end{equation}
This indicates that $\calW_\calS (M)$ is already a kind of ``formal deformation quantization'', but just 
too big. It was Fedosov's fundamental idea to construct an appropriate flat connection $D$ on 
$\calW_\calS (M)$ such that the subalgebra of flat sections, i.e.~of sections $a$ such that $Da =0$, is 
isomorphically mapped by the symbol map onto $\calC^\infty (M)[[\hbar]]$ and thus induces a star product 
on $\calC^\infty (M)[[\hbar]]$.  Let us explain Fedosov's construction of $D$.

We chooses  a Poisson connection $\nabla$ according to Prop.~\ref{PropPoissonConn}, which canonically 
lifts to a connection 
\[
  \nabla : \Omega^\bullet \fWeyl_\calS (M ) \rightarrow \Omega^{\bullet + 1} \fWeyl_\calS ( M ).
\]
Fedosov \cite[Sec. 5.2]{FedDQIT} proved that there exists a section $A \in \Omega^1 \calW_\calS (M)$ such that the connection
\begin{equation}
  \label{eq:FedConn}
   D : = \nabla + \frac{i}{\hbar} [ A, - ]_\circ
\end{equation}
is abelian, i.e.~satisfies $D \circ D = 0$. 
%The $1$-form $A$ is even uniquely 
%determined by the latter property, if one additionally requires that 
%$\deg_\textup{F} (A) \geq 2$. 
Such an abelian connection $D$ defined by a  $1$-form $A$
will be called a \emph{Fedosov connection}.  

We briefly explain the uniqueness of the star product.  Let $\{x^1, \cdots, x^{\operatorname{rk}(\Pi)}\}$ be 
leafwise coordinates along $\calS$, and $\{y^1, \cdots, y^{\operatorname{rk}(\Pi)}\}$ be the dual elements in 
$T^*\calS$. Define $\delta: \Omega^\bullet \fWeyl_\calS(M)\to \Omega^{\bullet+1}\fWeyl_\calS(M)$ and 
$\delta^*:\Omega^\bullet\fWeyl_\calS(M)\to \Omega^{\bullet-1}\fWeyl_\calS(M)$ by 
\[
  \delta a=\sum_{k=1}^{\operatorname{rk}(\Pi)}dx^k\wedge\frac{\partial a}{\partial y^k}, \qquad 
  \delta^*a= \sum_{k=1}^{\operatorname{rk}(\Pi)}y^k\: \iota_{\! \frac{\partial}{\partial x^k}}a. 
\]
%Direct computation shows that the connection $\nabla$ on $\Omega^\bullet \fWeyl_\calS(M)$ satisfies 
%\[
%\nabla^2=\frac{i}{\hbar}[R_\nabla,-]_\circ,
%\]
%where $R_\nabla\in \Omega^2\calW_s(M)$ is determined by the original Poisson connection $\nabla$ on $TS$, c.f. Prop. \ref{PropPoissonConn}. Given a connection $D$ of the form (\ref{eq:FedConn}), define $\Omega_{\nabla, A}=R_\nabla+\nabla A+\frac{i}{\hbar}A\circ A$, called the curvature of the connection $D$. 
Given an abelian connection $D$ of the form (\ref{eq:FedConn}), 
direct computation shows that there is a canonical element $\Omega_D \in \calW_\calS(M)$, called the 
\emph{curvature of $D$}, associated to the Poisson connection $\nabla$ (cf.~Prop.~\ref{PropPoissonConn}) and $A$ such that 
\[
D^2=\frac{i}{\hbar}[\Omega_{D}, -]_\circ .
\]
Let $\Omega^\bullet_\calS(M, \mathbb{R}[[\hbar]])$ be the space of leafwise differential forms along $S$ with coefficient in $\mathbb{R}[[\hbar]]$. As $D^2=0$, $\Omega_{D}$ is in the center of $\calW_\calS(M)$, and therefore an element in $\Omega_\calS^2(M, \mathbb{R}[[\hbar]])$ closed under the de Rham differential.  
{\sc Fedosov} \cite[Thm.~5.2.2]{FedDQIT} proved that under the requirements
\begin{enumerate} 
\item $\deg_\textup{F} (A) \geq 2$,
\item $\delta^* A=0$,
\end{enumerate} 
there is a unique Fedosov connection $D$ associated to a given Poisson connection $\nabla$
which has the given curvature form $\Omega$. In what follows, we will always assume to work with Fedosov connections with the above assumptions. 
%The connection $D$ acts on $\Omega^\bullet\fWeyl_s(M)$ by derivations, i.e. for $a,b\in \Omega^\bullet \fWeyl_s(M)$.
%\[
%D(a\circ b)=D(a)\circ b+(-1)^{deg(a)}a\circ D(b). 
%\] 

%As has been observed by Fedosov \cite{FedDQIT}, 

Let us fix a Fedosov connection $D$ and consider the space 
\[
 \calW_D (M) := \{ a \in \calW_\calS (M) \mid Da = 0 \}
\]
of flat sections of the Weyl algebra bundle. $\calW_D(M)$ is a subalgebra of $\calW |_\calS M$ as $D$ is 
a derivation. %gives rise to  a deformation quantization of $\calC^\infty (M)$ via
%the symbol map
%\[
  %\sigma :  \calW_\calS (M) \rightarrow \calC^\infty (M)[[\hbar]],  \quad a = \sum_{s\in \N, k \in \N} a_{s,k} \hbar^k \mapsto
   %\sum_{k \in \N} a_{0,k} \hbar^k .
%\]
%More precisely, if the $1$-form $A$ has been chosen as above, 
{\sc Fedosov} \cite{FedDQIT} observed that the restriction of the symbol map (\ref{eq:symbol})
\[
  \sigma_{|\calW_D (M)} :  \calW_D (M) \rightarrow \calC^\infty (M)[[\hbar]]
\]
is a linear isomorphism. Let 
\[
  \frq :   \calC^\infty (M)[[\hbar]]  \rightarrow \calW_D (M)
\]
be its inverse, the so-called \emph{quantization map}. 
Then there exist uniquely determined differential operators 
$\frq_k :  \calC^\infty (M) \rightarrow  \calC^\infty (M)$ such that 
\begin{equation}
\label{eq:expquantmap}
   \frq (f) = \sum_{k\in \N} \frq_k (f) \hbar^k \quad 
   \text{for all $f \in \calC^\infty (M)$},
\end{equation}
and 
\[
  \star : \calC^\infty (M)[[\hbar]]  \times \calC^\infty (M)[[\hbar]] , \quad
  (f , g) \mapsto \sigma \big(  \frq (f) \circ \frq (g) \big)
\] 
is a star product on $\calC^\infty (M)$.

Now observe that the Fedosov connection $D$ leaves the module 
$\calJ^\infty (X;M) \cdot \Omega^\bullet (M;\fWeyl_\calS M)$ invariant.   
This implies that $D$ factors to the quotient 
\[ 
  \Omega^\bullet_{\calE^\infty} (X;\fWeyl_\calS M) := 
  \Omega^\bullet (M;\fWeyl_\calS M) / \calJ^\infty (X;M) \cdot \Omega^\bullet (M;\fWeyl_\calS M),
\]
and acts on 
$\calE^\infty (X;\fWeyl_\calS M) := \calW_\calS ( M) / \calJ^\infty (X;M) \cdot \calW_\calS  (M)$. 
Moreover, the symbol map $\sigma$ maps  
$\calJ^\infty (X;M)\cdot \calW (M)$ to $\calJ^\infty (X;M)[[\hbar]]$, and 
$\frq \big( \calJ^\infty (X;M)[[\hbar]]\big)$ is contained in $\calJ^\infty (X;M)\cdot \calW (M)$,
since in the expansion \eqref{eq:expquantmap} the operators $\frq_k$ are all differential operators.  
Hence $\sigma$ and $\frq$ factor to $\calE^\infty (X;\fWeyl M)$ respectively 
$\calE^\infty (X)[[\hbar]]$. This entails the following result, which generalizes
\cite[Thm.~1.5]{PflPosTanQWF} to the regular Poisson case. 

%\begin{theorem}
%  Let $(M,\Pi)$ be a regular Poisson manifold, $D$ a Fedosov connection on $\Omega^\bullet \fWeyl_\calS $,
%  and $X\subset M$ a closed subset. Then the space of flat sections 
%  \[
%    \calW_D (X) := \{ a \in \calE^\infty (X;\fWeyl_\calS M)  \mid Da = 0\}
%  \] 
%  is a subalgebra of $\calE^\infty (X;\fWeyl_\calS M)$, and the symbol map induces an isomorphism of linear 
%  spaces 
%  $\sigma_X : \calW_D (X) \rightarrow \calE^\infty(X)[[\hbar]]$. Moreover, the unique product $\star$ on 
%  $\calE^\infty(X)[[\hbar]]$ with respect to which $\sigma_X$ becomes an isomorphism of algebras 
%  is  a formal deformation quantization of $\calE^\infty(X)$.
%\end{theorem}

\begin{theorem}\label{thm:starprod}
  Let $(M,\Pi)$ be a regular Poisson manifold, and $\nabla$ a Poisson connection. 
  Let $D = \nabla + A $ the corresponding Fedosov connection on $\Omega^\bullet \fWeyl_\calS $,
  and $X\subset M$ a closed subset. Then the space of flat sections 
  \[
    \calW_D (X) := \{ a \in \calE^\infty (X;\fWeyl_\calS M)  \mid Da = 0\}
  \] 
  is a subalgebra of $\calE^\infty (X;\fWeyl_\calS M)$, and the symbol map induces an isomorphism of 
  linear spaces $\sigma_X : \calW_D (X) \rightarrow \calE^\infty(X)[[\hbar]]$. 
  Moreover, the unique product $\star_X$ on 
  $\calE^\infty(X)[[\hbar]]$ with respect to which $\sigma_X$ becomes an isomorphism of algebras 
  is  a formal deformation quantization of $\calE^\infty(X)$.
\end{theorem}

By the uniqueness property of the Fedosov connection with respect to the curvature form $\Omega_{D}$, 
we have the following functoriality property of the start products constructed in 
Thm.~\ref{thm:starprod}. 
\begin{proposition}\label{prop:functoriality}
  The Fedosov quantization of Whitney functions on closed subspaces of regular Poisson manifolds is 
  functorial in the following sense. 
  Let $\Phi : (N,\Lambda) \rightarrow (M,\Pi)$ be a Poisson map
  between regular Poisson manifolds which maps the closed subset 
  $Y\subset N$  to the closed subset $X\subset M$.  Assume that the restriction of $\Phi$ to each 
  symplectic leaf of $\Lambda$ is a (local) diffeomorphism, and further that $\nabla^N$ and $\nabla^M$ 
  are Poisson connections on $N$ respectively $M$ such that  
  $\nabla^N = \Phi^* \big( \nabla^M \big)$. Denote by $\calS$ the symplectic foliation on $M$, by 
  $\calR$ the symplectic foliation on $N$.
  Let $D^N$ resp.~$D^M$ be the corresponding Fedosov connection  with the curvature form $\Omega_{D^N}$ 
  resp.~$\Omega_{D^M}$ and the induced star product 
  $\star_Y$ resp.~$\star_X$. Assume that $ \Omega_{D^N}=\Phi^*\big(\Omega_{D^M}\big)$. Then the pullback 
  $\Phi^*:\fWeyl_\calS(M)\to \fWeyl_\calR (N)$ is an algebra morphism
  \[
    \Phi^*  :  \big( \calE^\infty (X;M) , \star_X \big) \rightarrow \big( \calE^\infty (Y;N),\star_Y \big)
  \]
 which is functorial and contravariant in $\Phi$ with the above mentioned properties. 
  
%  \\[2mm]
%  Conjecture: $\frak{Q} (F) = \Phi^*$?
\end{proposition}
\begin{proof}
Since $\Phi$ restricts to a (local) symplectic diffeomorphism between symplectic leaves, it is 
straightforward to check that the pullback map $\Phi^*: T^*M\to T^*N$ lifts to a morphism of the 
corresponding Weyl algebra bundles,
\[
  \Phi^*: \fWeyl_\calS (M)\to \fWeyl_\calR (N).
\]
As $\Phi$ is assumed to be compatible with the Poisson connections, 
i.e.~$\nabla^N = \Phi^* \big( \nabla^M \big)$, and also the curvature forms, 
i.e.~$\Omega_{D^N}=\Phi^*\big(\Omega_{D^M}\big)$, the uniqueness property of the Fedosov connection 
with respect to the curvature form and Poisson connection implies that 
\[
  D^N\circ \Phi^*=\Phi^*\circ D^M.
\] 
Hence, $\Phi^*$ restricts to an algebra morphism 
\[
  \Phi^*: \calW_{D^M}(M)\to \calW_{D^N}(N), 
\]
and therefore a morphism
\[
  \Phi^* : \big( \calE^\infty (X;M) , \star_X \big) \rightarrow \big( \calE^\infty (Y;N),
  \star_Y \big).
\]
\end{proof}
\section{Whitney functions on an orbit space and their quantization}
Assume that $G$ is a compact Lie group acting freely on the smooth manifold $M$, and denote by  
$\pi : M\rightarrow N$ the canonical projection onto the orbit space $N:= M/G$ which under our
assumption is a smooth  manifold as well. 
Let $X \subset M$ be a closed $G$-invariant subset, and $Y:=X/G$. Then $Y$ is a closed subset of $N$.
Under these assumptions, the following result holds true.
\begin{proposition}
  The canonical projection induces a natural identification
  \[
     \pi^* : \calE^\infty (Y ; N) \cong \calE^\infty (X;M)^G .   
  \]
  Here, $\calE^\infty (X;M)^G$ denotes the set of Whitney functions
  represented by $G$-invariant smooth functions, i.e.~the image of the space $\big(\calC^\infty (M)\big)^G $
  of $G$-invariant smooth functions on $M$ under the jet map $\jet_{X;M}$.
\end{proposition}

\begin{proof}
  Observe first that the image of $\pi^*$ lies in $\calE^\infty (X;M)^G$ indeed by definition of the 
  pull-back of Whitney functions and since $f\circ \pi$ is $G$-invariant for every 
  $f\in \calC^\infty (N)$. Since $\pi$ is a surjective, the pull-back 
  $\calC^\infty (N) \rightarrow  \big( \calC^\infty (M) \big)^G$, $f \mapsto f\circ \pi $
  is injective. Hence $\pi^* :  \calE^\infty (Y ; N) \rightarrow \calE^\infty (X;M)^G$
  is injective as well, if we can yet show that $f\circ \pi \in \calJ^\infty (X,M)$ for
  $f\in \calJ^\infty (Y,N)$. But this follows from the multidimensional Fa\`a di Bruno formula, 
  cf.~\cite[Thm.~3.6]{MichTG}. More precisely, this formula says that for 
  $x\in X$, a coordinate system $(x_1, \ldots , x_n)$ around $x$, a coordinate system 
  $ (y_1, \ldots , x_m)$ around $\pi (x)$, and a multiindex $\gamma \in \N^n$ 
  the following equality holds true:
  \begin{align*}
    \partial^\gamma (f\circ \pi) = \!\!\!\!
    \sum_{\lambda = (\lambda_{i,\alpha}) \in \N^{m \times {\N^n \setminus \{ 0 \} }}
    \atop \sum \lambda_{i,\alpha} \alpha = \gamma} 
    \frac{\gamma !}{\lambda !} \: \prod_{\alpha \in \N^n \atop |\alpha | > 0 } \!\!\!\!
    \left(  \frac{1}{\alpha !}   \right)^{\sum_i \lambda_{i,\alpha}}
    \left(\partial^{\sum_{\alpha}(\lambda_{1,\alpha} ,\ldots , \lambda_{m,\alpha}}) f\right)\circ \pi \:
    \prod_{i,\alpha} \left(\partial^\alpha \pi_i \right)^{\lambda_{i,\alpha}}, 
  \end{align*}
  where $\pi_i$ denotes the $i$-th component function of $\pi$ (in a neighborhood $x$) 
  with respect to the coordinate system $y$ around $\pi (x)$. This implies that if all
  $\partial^{\sum_{\alpha}(\lambda_{1,\alpha} ,\ldots , \lambda_{m,\alpha}}) f$ vanish on $Y$ then
  $\partial^\gamma (f\circ \pi)$ vanishes on $X$. Hence $f\in \calJ^\infty (Y,N)$
  implies $f\circ \pi \in \calJ^\infty (X,M)$, and $\pi^*$ is injective. 
  Surjectivity of $\pi^*$ follows from the Theorem by {\sc Schwarz--Mather} \cite{SchwaSFIACLG,MatDI} 
  which in particularly  says that the map
  \[
    \calC^\infty (N) \rightarrow \big( \calC^\infty (M)\big)^G  , \: f \mapsto f \circ \pi 
  \]
  is split-surjective.
\end{proof}

\begin{remark}
  This result has been proven in the general case without the restriction of the $G$-action to be 
  free in \cite{HerPfl}. 
\end{remark}

Next we choose a $G$-invariant Poisson connection $\nabla$ on $M$ according to Thm.~\ref{PropPoissonConn}.
Let us also fix the $G$-invariant curvature form $\Omega = -\omega$, where $\omega$ denotes the
fiberwise symplectic form on $TS$. Then, by the preceeding section, there exists 
a uniquely determined Fedosov connection $D$ having the given curvature form $\Omega$. 
By construction, the connection $D$ is $G$-invariant as well. Let $\star$ denote the 
corresponding star product on $\calC^\infty (M)[[\hbar]]$. By invariance of $D$, the star product $\star$
is invariant as well, which means that for two $G$-invariant functions $f,g\in \calC^\infty (M)^G$ their star product 
$f\star g$ is also $G$-invariant. This observation together with the previous proposition entail the first two 
claims of the following result.

\begin{theorem}
\label{MainThm}
  The Fedosov star product $\star$ associated to a $G$-invariant Poisson connection $\nabla$ on $M$ (and to the 
  $G$-invariant curvature form $\Omega = -\omega$) is $G$-invariant, hence
  \begin{equation}
    \label{Eq:InvStarProdAlg}
    \Big( \big( \calE^\infty (X)[[\hbar]] \big)^G, \star \Big) 
  \end{equation}
  is a subalgebra of $\Big( \calE^\infty (X)[[\hbar]] , \star \Big)$. Moreover, under
  the  isomorphism $\pi^* : \calE^\infty (Y ; N) \cong \calE^\infty (X;M)^G $ one obtains
  a star product algebra
  \[
     \Big( \calE^\infty (Y)[[\hbar]] , \overline{\star} \Big) ,
  \]
  where $F\,\overline{\star}\, G$ for $F,G \in \calE^\infty (Y)$ is defined by  $(\pi^*)^{-1} \big( \pi^*(F) \star \pi^* (G)\big)$. 
  Finally, if $(N,\lambda)$ is a regular Poisson manifold, then  $\Big( \calE^\infty (Y)[[\hbar]] , \overline{\star} \Big)$
  is isomorphic to the Fedosov deformation quantization $(\calE^\infty (Y)[[\hbar]] , \star_{\nabla^N} )$
  corresponding to a Poisson connection $\nabla^N$ on $N$ and to the curvature form $-\omega^\calR$, 
  where $\omega^\calR$ denotes the leafwise symplectic form on the symplectic foliation $\calR$ of $N$.  
\end{theorem}

\begin{remark}
  The last statement of the theorem is a ``quantization commutes with reduction'' result for quantized 
  Whitney functions. 
  
  Note that in general, the Poisson manifold $N$ needs not be regular, hence the above theorem provides a  
  quantization method for Whitney functions on subsets of not necessarily regular Poisson manifolds which can be written
  as the quotient of a regular Poisson manifold  by a compact Lie group action. 
\end{remark}
Before proving the theorem, let us state some results needed in the proof. 
\begin{proposition}
\label{SympLinGeo}
  Let $(V,\omega)$ be a presymplectic vector space and $W\subset V$ a linear subspace. Then the following
  equality holds true:
  \[
    \dim W + \dim W^\omega = \dim V + \dim (W \cap V^\omega).
  \]
  Furthermore, if $\omega$ is non-degenerate and $W$ is symplectic, then $W^\omega$ is symplectic as well.
\end{proposition}
\begin{proof}
  This is  a straightforward argument in linear symplectic geometry. 
\end{proof}

\begin{lemma}
  Any element $g\in G$ maps symplectic leaves of $M$ to symplectic leaves.  
\end{lemma}

\begin{proof}
  Let $L \subset M$ be a symplectic leaf with symplectic form $\omega$. 
  Consider the connected submanifold $gL \subset M$, and two points $x,y \in gL$.
  Since $\Pi$ is $g$-invariant, the restriction $\Pi_{|gL}$ is a Poisson bivector on $gL$ of maximal rank,
  and its corresponding symplectic form coincides with $g_*\omega$. It remains to show that
  $x$ and  $y$ can be connected by a piecewise smooth curve whose smooth parts are integral curves of 
  Hamiltonian vector fields. But this is clear, since $g^{-1}x$ and $g^{-1}y$   are both elements 
  of the symplectic leaf $S$, hence can be connected within $L$ by a piecewise  smooth curve $\gamma$ whose 
  smooth parts are integral curves of Hamiltonian vector fields. The curve $g \gamma$ then connects
  $x$ and $y$ and has the desired properties by $G$-invariance of $\Pi$.
\end{proof}

\begin{proposition}
\label{IsoLeaf}
  For every symplectic leaf $L\subset M$ there exists a closed subgroup $H_L \subset G$ called the 
  \emph{isotropy group of $L$}  which leaves
  $L$ invariant and which has the property  that for each point $x\in L$ the fiber 
  $\pi^{-1} \big(\pi (x)\big)$ coincides with the orbit $H_Lx$. In other words, one has the natural 
  isomorphism $\pi (L) \cong L/H $.
\end{proposition}

\begin{proof}
  By the preceeding lemma, the group $G$ acts on the space $Z$ of symplectic leaves of $M$. Let $H_L$ be 
  the 
  isotropy group of the point $L\in Z$. Clearly, $H_L$ then is a closed subgroup of $G$ and has the desired
  properties. 
\end{proof}

\begin{proof}[Proof of Thm.~\ref{MainThm}]
It only remains to prove the last claim which says that the star product algebras 
$\big( \calE^\infty (Y)[[\hbar]] , \star_{\nabla^N} \big)$ and 
$\big( \calE^\infty (Y)[[\hbar]] , \overline{\star} \big)$ are isomorphic when $N$ is regular Poisson. 
For this we use the well-known result \cite{DelDAFVS,FedDQIT,neumaier,BurDolWal} that on the regular 
Poisson manifold $N$, two deformation quantizations $\star$ and $\star'$ are isomorphic if and only if 
they have the same characteristic class in the formal cohomology $\omega/\hbar+H^2_\calS (N,\C[[\hbar]])$, 
where $\calS$ denotes the symplectic foliation. 
Precisely, this means that there exists a formal power series $G=1+\hbar D_1+\ldots$ of differential 
operators tangent to the leaves of $\calS$ such that 
\[
G^{-1}\left( G(f_1)\star G(f_2)\right)=f_1\star' f_2.
\]
Obviously, $G$ preserves the ideal $\calJ^\infty(Y;N)[[\hbar]]$, so it induces an isomorphism between $(\calE^\infty(Y),\star)$ and $(\calE^\infty(Y),\star')$. Therefore, the claim follows from the fact that both 
$\big( \calE^\infty (Y)[[\hbar]] , \star_{\nabla^N} \big)$ and 
$\big( \calE^\infty (Y)[[\hbar]] , \overline{\star} \big)$ 
have the same characteristic class, namely $\omega^\calR/\hbar$.

So finally it remains to prove that the characteristic class of $\overline{\star}$ is 
$\omega^\calR/\hbar$, indeed (for every initially chosen $G$-invariant Poisson connection $\nabla^M$
and every Poisson connection $\nabla^N$). To this end, it suffices to prove this claim for a particular
choice of $\nabla^M$ and $\nabla^N$. Fix a Poisson connection  $\nabla^M$. We first want to construct 
a ``compatible'' Poisson connection $\nabla^N$. 

Since $M$ is foliated into symplectic leaves and the connections act leafwise, it suffices to
prove the claim for each leaf separately. Due to Prop.~\ref{IsoLeaf} we can therefore 
assume  without loss of generality, that $M$ is symplectic, and $G$ acts by symplectomorphisms on $M$.
To prove the claim, we will decompose the tangent bundle $TM$ in appropriate $G$-invariant 
subbundles which then will allow a unique lift of vector fields on $N$ tangent to the symplectic 
foliation $\calR$ of $N$ to invariant vector fields on $M$ having values in a certain subbundle. 

To this end 
let $G'$ be the standard polar pseudogroup associated to $G$ as defined in {\sc Ortega--Ratiu} 
\cite[Sec.~5.5.1]{OrtRatMMHR}. In other words, $G'$ is the pseudogroup of local diffeomorphisms of $M$ 
generated by the flows of Hamiltonian vector fields of the form
 $X_f := \Pi \lrcorner df : U \rightarrow TM$, where $f \in \big( \calC^\infty( U)\big)^G$ and $U$ is a 
$G$-invariant open subset of $M$. According to 
\cite[Sec.~5.5.1 \& Thm.~11.4.4]{OrtRatMMHR}, the actions by $G$ and $G'$ commute, and the symplectic 
leaves of $M/G$ are given by the (piecewise) orbits of the induced $G'$-action on $M/G$. Let $E$ 
be the vector bundle generated by such (invariant) Hamiltonian vector fields $X_f$. Then $E$ together with 
the restriction of the symplectic form $\omega$ to $E$ is a pre-symplectic bundle over $M$. 
By construction, the bundle $E$ is mapped under $T\pi$ onto the tangent bundle $T\calR$ of the symplectic foliation 
of $N$. Moreover, 
\begin{equation}
\label{Eq:SympCompOrb}
   E \subset T \calO ^\omega ,
\end{equation}
since one has for every $w \in E_p$, $p\in M$ and every fundamental $X_\xi$ of an element $\xi \in \frakg$ the 
relation
\[
   \omega \big( w , X_\xi (p) \big) = \omega \big( X_f(p) , X_\xi (p) \big) =  (X_\xi f) (p) = 0 ,
\]
where the $G$-invariant smooth function $f$ on $M$ has been chosen such that  $w = X_f (p)$. 
Now choose a $G$-invariant riemannian metric $\eta$ on $M$, and let $W$ be the  orthogonal complement of 
$T\calO \cap E$ in $E$, where $\calO$ denotes the foliation of $M$ by the $G$-orbits. 
By the regularity assumption on the induced Poisson structure on $N$
it is clear that $W$ is a vector bundle indeed. By construction, $W$ is a $G$-invariant subbundle of $E$ complementary 
to $E\cap T\calO$. Since $T\calO$ is the kernel bundle of the tangent map of the projection, $T\pi$, it follows that 
$T\pi$ maps $W$ onto the tangent bundle $T\calR$ of the symplectic foliation in such a way that fiberwise, 
$T\pi_{|W} : W \rightarrow T\calR$ is a linear symplectic isomorphism. 
This observation allows us to construct for every vector field $X$ on $N$ which is tangent to
$\calR$ a unique lift $X^*:M\rightarrow W$ such that 
\[
  T\pi \, X^* (p) = X(\pi(p)) \quad \text{for all $p\in M$}.
\]
Now we can define a connection $\nabla^N$ on $T\calR$ by putting, for any two vector fields
$X,Y$ on $N$ tangent to the symplectic foliation $\calR$,
\[
  \nabla^N_X Y : = T\pi \, \nabla^M_{X^*} Y^* \: .
\]
Clearly, $\nabla^N$ is torsion-free, so we only need to check that $\nabla^N$ is a Poisson connection.
For $X,Y$ as before let $A_{X,Y} : M \rightarrow TM$ be the vector field
\[
  A_{X,Y} := \big( T\pi \, \nabla^M_{X^*} Y^* \big)^* - \nabla^M_{X^*} Y^*\: .
\]
By construction, $ A_{X,Y} (p) \in T_p\calO$ for all $p\in M$. This gives for the leafwise
symplectic form $\omega^\calR$ on $T\calR$ and smooth vector fields $X,Y,Z$ on $N$ tangent to $\calR$:
\[
\begin{split}
  Z \big( \omega^\calR & (X,Y)\big) (\pi(p)) \\ = \, & Z^* \big( \omega (X^*,Y^*)\big) (p) =
  \omega\big( \nabla^M_{Z^*}X^* , Y^* \big) (p) +  \omega\big( X^*, \nabla^M_{Z^*} Y^* \big) (p) = \\
  = \, & 
  \omega\big( (\nabla^N_{Z}X)^* , Y^* \big) (p) +  \omega\big( X^*, (\nabla^N_{Z^*} Y)^* \big) (p) + \\
  & + \omega\big( A_{Z,X}  ,  Y^* \big) (p) + \omega\big( X^* , A_{Z,Y} \big) (p) = \\
  = \, & 
  \omega^\calR \big( \nabla^N_{Z}X , Y \big) (\pi(p)) +  \omega^\calR \big( X, \nabla^N_{Z^*} Y \big) (\pi(p)),
\end{split}
\]
where the last equality follows from the fact that the vector fields $ A_{Z,X} $  and $ A_{Z,Y} $ 
are tangent to the orbit direction, and that the lifted vector fields $Y^*$ and $X^*$ lie in the symplectic 
orthogonal complement of the orbit direction by Eq.~\ref{Eq:SympCompOrb}. Hence, $\nabla^N$ is a Poisson connection.

Finally, observe that the leafwise symplectic form $\omega^\calR$ on $N$ and the symplectic form $\omega$
on $M$ are related by 
\[
  \omega (X^*,Y^*) (p) = \omega^\calR (X,Y) (\pi(p)),
\]
which implies that the induced Fedosov connections on $N$ and $M$ are related in an analogous fashion. 
This implies in particular that the characteristic classes of the star products 
$\overline{\star}$ and $\star_{\nabla^N}$ coincides in both cases with $\omega^\calR/\hbar$.
The proof is finished.
\end{proof}

\begin{remark}
  The proof of the theorem shows even more, namely that for the Poisson connection $\nabla^N$ constructed
  in the proof, the star products $\star_{\nabla^N}$ and  $ \overline{\star}$ 
  on $(\calE^\infty (Y)[[\hbar]]$ even coincide. 
\end{remark}

\begin{example}
  Let $(M,\omega,G,J)$ be a Hamiltonian system with free $G$-action, and 
  consider 
  the stratification of $\frakg^*$ with the coadjoint action by orbit types.
  Let $S_\circ \subset \frakg^* $ be the open dense stratum, and 
  put $U:= J^{-1} (S_\circ)$. Then the quotient $V:=U/G$ is 
  a regular Poisson manifold, and the above 
   ``quantization commutes with reduction'' result applies to any $G$-invariant
  closed $X \subset U$.
\end{example}
\section{outlook}
The results from the previous section indicate that methods of real algebraic geometry and singularity theory might 
be helpful in solving problems in Poisson geometry. In the follwoing list we describe some of the problems,
where we expect that combining methods from singularity theory with Poisson geometry could eventually lead to 
the solution of the outstanding questions. 

\begin{itemize}
\item Even though one can construct deformation quantizations of Whitney functions over singular sets as
      explained above, a full (deformation) quantization theory of algebras of smooth functions 
      over singular symplectic spaces is still lacking. Partial results exist, though, as  
      the papers \cite{BorHerPflHASRDQ,HerPflIyeESPQSLHTA} show, where deformation quantizations of a particular
      class of singular symplectically reduced spaces are constructed by homological perturbation theory. 
      More precisely, the algebra of smooth functions on the zero level set of a $G$-hamiltonian 
      system is resolved there by a Koszul complex  defined by the moment map (again, under certain assumptions 
      on the $G$-Hamiltonian system). The symmetry group $G$ acts in a natural way on the Koszul complex
      which allows to represent the algebra of smooth functions on the symplectically reduced 
      space as the cohomology group in degree $0$ of the so-called (classical) BRST complex. 
      Appropriately deforming the BRST complex eventually then gives rise to a deformation quantization on the 
      symplectically reduced space. Generalizing this idea, one expects that the Koszul resolution appearing in this construction
      needs to be replaced by a Koszul--Tate resolution having infinite length. Sophisticated methods from
      commutative algebra and singularity theory then might eventually lead to the construction of 
      star products on any symplectically reduced space.

\item There are certain no go theorems on the existence of embeddings of a given symplectic (or Poisson) stratified space 
      into a Poisson manifold, see \cite{EgiEBD,DavisEBD}. 
      It appears that methods from commutative algebra and singularity theory could share more 
      light on this phenomenon  and possibly will lead to a more precise characterization of the obstructions to such embeddings.  

\item Hochschild and cyclic homology theory of function algebras over singular spaces provide useful information on the existence
      of  deformations of these algebras, and are the essential ingrediants in the study of the underlying singular spaces
      by methods of noncommutative geometry invented by A.~Connes \cite{ConNDG}. Again, a synthesis of methods from singularity theory
      with those from differential geometry, and possibly even noncommutative geometry has already led to interesting results, see for example 
      \cite{nppt,PPT2ADV,PflPosTanQWF,PflPosTanMPOSPLG}, 
      and might lead to further new observations in either of these areas. Work on this is in progress, see \cite{HerPfl}.

\end{itemize}

\bibliographystyle{alpha}

\begin{thebibliography}{}

\bibitem[{\sc BFFLS}]{BayetalDTQIuII}
  {\sc Bayen, F., M.~Flato, C.~Fronsdal, A.~Lichnerowicz},
  and {\sc D.~Sternheimer:} 
  {\it Deformation theory and quantization, {$I$} and {$II$}}, 
  Ann.~Phys. {\bf 111} (1978), 61--151.

\bibitem[{\sc BoHePf}]{BorHerPflHASRDQ}
  { \sc  M.~Bordemann, H.-C. Herbig} and {\sc M.~Pflaum}:
  {\it A homological approach to singular reduction in deformation 
  quantization}, in {\it Singularity Theory} (Eds.~Ch\'eniot et.~al.),
  dedicated to Jean-Paul Brasselet on his 60th birthday,
  Proceedings of the 2005 Marseille Singularity School and Conference
  CIRM, Marseille, France 24 January - 25 February 2005,
  World Scientific (2007).

\bibitem[{\sc BrPf}]{BraPfl}
  {\sc J.-P.~Brasselet} and {\sc M.~Pflaum}:
  \textit{On the homology of algebras of Whitney functions over subanalytic
  sets}.
  Annals of Math.~\textbf{167}, 1--52 (2008).

\bibitem[{\sc BuDoWa}]{BurDolWal}
   {\sc Bursztyn, H., V.~Dolgushev}, and {\sc S.~Waldmann}:
   \emph{Morita equivalence and characteristic classes of star products}.
   J.~Reine Angew.~Math.~\textbf{662} (2012), 95--163. 

%\bibitem[{\sc Bry}]{BryDCPM}
%  {\sc J.-L. Brylinski}: \textit{A differential complex for Poisson manifolds}, 
%  J. Differential Geometry \textbf{28} (1988), 93--114.

\bibitem[{\sc Con}]{ConNDG}
  {\sc A.~Connes}: \emph{Noncommutative differential geometry}, Inst.~Hautes {\'E}tudes
  Sci. Publ. Math. \textbf{62} (1985), 257--360.

\bibitem[{\sc Dav}]{DavisEBD}
  {\sc B.~L.~Davis}:
  {\it Embedding dimensions of {P}oisson spaces},
  Int. Math. Res. Not. \textbf{34}, 1805--1839 (2002).

\bibitem[{\sc Del}]{DelDAFVS}
  {\sc P.~Deligne}: \emph{D\'eformations de l'Alg\`ebre des Fonctions d'une Vari\'et\'e Symplectique:
  Comparaison entre Fedosov et De Wilde, Lecomte}.
  Selecta Mathematica, New Series Vol.\textbf{1}, No.~4 (1995).

%\bibitem[{\sc Dol}]{DolgADV}
%  {\sc V.~Dolgushev}: \textit{A formality theorem for Hochschild chains},
%  Adv. Math. \textbf{200} (2006), no. 1, 51--101.

\bibitem[{\sc Egi}]{EgiEBD}
  {\sc A.~S.~Egilsson}:
  {\it On embedding the {$1\:1\:2$} resonance space in a {P}oisson manifold},
  Electron. Res. Announc. Amer. Math. Soc. \textbf{1}(2), 48--56
  (electronic), (1995).

\bibitem[{\sc Fed}]{FedDQIT} {\sc B. Fedosov}:
  \textit{Deformation quantization and index theory}. Akademie Verlag, 1995.

\bibitem[{\sc HeIyPf}]{HerPflIyeESPQSLHTA}
  {\sc H.-C.~Herbig, S.~Iyengar} and {\sc M.~Pflaum}:
  {\it On the existence of star products on quotient spaces of linear 
  Hamiltonian torus actions},
  Lett.~in Math.~Physics \textbf{89}, No.~2, 101--113 (2009).

\bibitem[{\sc HerPfl}]{HerPfl}
  {\sc Herbig, H.-Ch.}, and {\sc M.~Pflaum}: \textit{Hochschild Homology of Algebras of Smooth Functions on Orbit Spaces}.
  in preparation. 

\bibitem[{\sc Kon}]{KonDQPM}
 {\sc M.~Kontsevich:}
 {\it Deformation quantization of Poisson manifolds},
 Lett. Math. Phys. \textbf{66}, No.3, 157-216 (2003).

%\bibitem[{\sc Lod}]{Loday}
%  {\sc J.-L.~Loday}: \textit{Cyclic homology}.
%  Second edition. Grundlehren der Mathematischen Wissenschaften \textbf{301},
%  Springer-Verlag, Berlin, 1998.

\bibitem[{\sc Mat}]{MatDI}
  {\sc Mather, J.}, \emph{Differentiable invariants}, Topology \textbf{16} (1977), no.~2,
  145--155.

\bibitem[{\sc McMil}]{McMillanEBD}
 {\sc A.~McMillan}:
 {\it On Embedding Singular Poisson Spaces},
 \texttt{arXiv:1108.2207} (2011).

\bibitem[{\sc Mic}]{MichTG}
  {\sc P.~W.~Michor}, \emph{Transformation groups}, 
  Lecture Notes of a course in Vienna (1993, 1997), 94 pp. 

%\bibitem[{\sc NeTs}]{NesTsyAIT}
% {\sc R.~Nest} and {\sc B.~Tsygan}:
% \textit{Algebraic Index Theorem}, 
% Comm.~Math.~Phys.~\textbf{172}, 223--262 (1995).

\bibitem[{\sc Neu}]{neumaier}
   {\sc N.~Neumaier}:
   \emph{Local {$\nu$}-{E}uler derivations and {D}eligne's
              characteristic class of {F}edosov star products and star
              products of special type}.
   Comm.~Math.~Phys.~\textbf{230}, nr.~2 (2002), 271--288.

\bibitem[{\sc NePfPoTa}]{nppt}
  {\sc N.~Neumaier, M.~Pflaum, H.~Posthuma} and {\sc X.~Tang}:
  {\it Homology of of formal deformations of proper \'etale Lie groupoids},
  Journal f.~die reine und angewandte Mathematik \textbf{593} (2006). 

\bibitem[{\sc OrtRat}]{OrtRatMMHR}
  {\sc Ortega, J.-P.}, and {\sc T.~Ratiu}: \textit{Momentum maps and Hamiltonian reduction}.
  Progress in Mathematics, \textbf{222}. Birkh{\"a}user Boston, Inc., Boston, MA, 2004.

\bibitem[{\sc PPT10}]{PPT2ADV}
 {\sc M.J.~Pflaum, H.~Posthuma} and {\sc X.~Tang}: 
 \textit{Cyclic cocycles on deformation quantizations and higher index theorems},
 Adv. Math. \textbf{223} (2010), no. 6, 1958--2021.

%\bibitem[{\sc PPT09}]{PPTLMP} 
%  {\sc M.J.~Pflaum, H.~Posthuma} and {\sc X.~Tang}: 
%  \textit{On the algebraic index for Riemannian {\'e}tale groupoids},
%  Lett. Math. Phys. \textbf{90} (2009), no. 1-3, 287--310.

\bibitem[{\sc PPT12}]{PflPosTanQWF} 
  {\sc M.J.~Pflaum, H.~Posthuma} and {\sc X.~Tang}: 
  \textit{Quantization of Whitney functions}.
  Travaux math\'ematiques, Vol.~\textbf{XX}, 153--165 (2012).

\bibitem[{\sc PPT13}]{PflPosTanMPOSPLG}
  {\sc M.J.~Pflaum, H.~Posthuma} and {\sc X.~Tang}: 
  {\it Geometry of orbit spaces of proper Lie groupoids}.
  to appear in Crelle Journal. 

\bibitem[{\sc Schwa}]{SchwaSFIACLG}
  {\sc Schwarz, G.W.}: \emph{Smooth functions invariant under the action of a compact
  {Lie} group}, Topology \textbf{14} (1975), 63--68.

\bibitem[{\sc Tel}]{Tlocal}
 {\sc N.~Teleman}: \textit{Microlocalisation de l'homologie de Hochschild},
 C. R. Acad. Sci. Paris S{\'e}r. I Math. \textbf{326} (1998), no. 11, 1261--1264.

\bibitem[{\sc Vai}]{VaiLGPM} 
  {\sc Vaisman, I.}:
  \textit{Lectures on the geometry of Poisson manifolds}.
  Progress in Mathematics, \textbf{118}. Birkh{\"a}user Verlag, Basel, 1994.
\end{thebibliography}

\end{document}